\documentclass[11pt,a4paper]{amsart}
\usepackage{amscd,amssymb,amsmath,latexsym,graphicx,epsfig,color,url}
\usepackage[active]{srcltx}
\usepackage[all]{xy}
\usepackage{hyperref}
\hypersetup{breaklinks=true}

\newtheorem{theorem}{Theorem}[section]
\newtheorem{corollary}[theorem]{Corollary}
\newtheorem{lemma}[theorem]{Lemma}
\newtheorem{proposition}[theorem]{Proposition}

\theoremstyle{definition}

\newtheorem{example}[theorem]{Example}
\newtheorem{remark}[theorem]{Remark}

\newcommand{\FF}{\mathbb F}

\DeclareMathOperator{\diag}{diag}

\newenvironment{matriz}[1]{\left[ \begin{array}{#1}}{\end{array} \right]}

\begin{document}

\title{The centralizer of an endomorphism over an arbitrary field}

\author{David Mingueza}
\email{david.mingueza@outlook.es}
\address{Accenture, Passeig Sant Gervasi 51-53, 08022 Barcelona, Spain}

\author{M.~Eul\`{a}lia Montoro}
\email{eula.montoro@ub.edu}
\address{Departamento de Matem\'aticas e Inform\'atica, Universitat de Barcelona, \\ Gran Via de les Corts Catalanes 585,
08007 Barcelona, Spain}

\author{Alicia Roca}
\email{aroca@mat.upv.es}
\address{Departamento de Matem\'{a}tica Aplicada, IMM, Universitat Polit\`ecnica de Val\`{e}ncia,\\ Camino de Vera s/n, 46022 Val\`encia, Spain}


\thanks{Montoro is supported by the Spanish MINECO/FEDER research project MTM 2015-65361-P and MTM2017-90682-REDT, 
Roca is supported by grant MTM2017-83624-P MINECO and MTM2017-90682-REDT. }



\begin{abstract}


The centralizer of an endomorphism of a finite dimensional vector space is known when the endomorphism is 
nonderogatory or when its minimal polynomial splits over the field. It is also known for the real Jordan canonical form. In this paper we characterize the centralizer of an  endomorphism over an arbitrary field, and compute its dimension. 
The result is obtained via generalized Jordan canonical forms (for separable and non separable minimal polynomials). 
In addition, we also obtain the corresponding generalized Weyr canonical forms and the structure of its  centralizers, which in turn allows
us to compute the determinant of its elements.

\end{abstract}

\maketitle

\section{Introduction}\label{sec:Intro}

The centralizer of an endomorphism has been widely described when the minimal polynomial splits on the underlying field, and different characterizations have been provided depending on the representation of the endomorphism. For the Jordan canonical form, a description of the centralizer can be found
in \cite{Gant1, Goh, Supru68} and for a Weyr canonical form, in \cite{Omeara}. 
For nonderogatory matrices  over arbitrary fields few references exist; a parametrization of the centralizer for companion matrices (i.e,, nonderogatory matrices) with irreducible minimal polynomial is given in \cite{Calde1}.
For derogatory matrices, a description of the centralizer over the real field is provided in \cite{Goh}. 
No general description  can be found  in the literature for the centralizer of derogatory matrices over arbitrary fields.

\medskip

When the minimal polynomial of a matrix has irreducible factors of degree greater 
than 1, it can not be reduced to Jordan or Weyr canonical forms. However, 
the Jordan canonical form admits different generalizations over arbitrary fields (rational canonical forms) depending on whether 
the minimal polynomial is separable (\cite{BrickFill65,Ro64, Ja75, Per70}) or nonseparable (\cite{Ro64,Holtz00,Dalalyan2014}). We will describe 
the centralizer in both cases. The generalized Jordan form is chosen because it allows us to find a parametrization of the centralizer, in a relatively simple way.

\medskip

We  first recall the centralizer of a companion matrix, building block of the  rational canonical forms. Then, we prove some technical lemmas, which solve 
certain matrix equations involving companion matrices. These results allow us  to obtain the centralizer for the generalized Jordan canonical form. Afterwards, 
adapting appropriately the technical lemmas, we  derive the results needed to obtain the centralizer for the separable case.

\medskip

From the generalization of the Jordan canonical form we derive the generalized Weyr canonical form, and obtain the corresponding centralizer (which, as far as we know, cannot be found in the literature).  Out of it, we also obtain an explicit formula for the determinant of the matrices in the centralizer. This fact is important in order to recognize the automorphisms of the centralizer, which is key, for instance, to study the hyperinvariant and characteristic lattices of an endomorphism (see~\cite{Ast1,MMP13,MMR18}). 
We also compute the dimension of the centralizer.

\medskip

The paper is organized as follows: in Section \ref{sec:preli} we recall some definitions, previous results and
the generalized Jordan canonical forms over arbitrary fields.
In Section \ref{sec:Weyr} we obtain the generalized Weyr canonical form.
Section \ref{gJm} is devoted to obtain the centralizer of the generalized Jordan form over arbitrary fields.
In Section \ref{sec:Perfect} we find the centralizer of the generalized Jordan form when the minimal polynomial is separable, for this case is not a particular case of the general one.
In Section \ref{sec:centraWeyr} we obtain the centralizers of matrices in the generalized Weyr canonical form,  compute the determinant of a matrix in the centralizer and, finally,  find the dimension of the centralizer.

\section{Preliminaries}\label{sec:preli}

\medskip
We recall some  definitions and previous results, which will be used throughout the paper.

\medskip

Let $V$ be a finite dimensional vector space over a field $\mathbb{F}$ and  $f:V \rightarrow V$ an
endomorphism. We denote by $A$ the matrix associated to $f$ with respect to a given basis, $p_{A}$ is the characteristic polynomial and 
$m_A$ is the minimal polynomial of $A$.
In what follows we will identify $f$ with $A$. The degree of a polynomial $p$ is written as $\deg(p)$.

\medskip

Given a matrix $A=[a_{i,j}]_{i, j=1, \ldots, n} \in M_{n}(\mathbb{F})$, we denote  
by $A_{\ast j}=\left[\begin{array}{c} a_{1 j}\\ \vdots \\ a_{ n j}\end{array}\right]$ the $j$-th column of $A$ and by  
$A_{i\ast}=\left[\begin{array}{ccc} a_{i1} & \ldots & a_{in} \end{array}\right]$ the $i$-th row of $A$, i.e.,  
$A=[A_{\ast 1 },\ldots,A_{\ast n}]$
and  $A=\left[\begin{array}{c} A_{1 \ast}\\ \vdots \\ A_{ n \ast}\end{array}\right]$.

\medskip

We recall the primary decomposition theorem, which establishes that a matrix $A \in M_{n}(\mathbb{F})$ is similar to a direct 
sum of matrices whose minimal polynomials are powers of distinct irreducible polynomials over $\mathbb{F}$.

\begin{theorem}[\cite{HoffKun71}, see also~\cite{BrickFill65, Per70}]
Let   $m_{A} = p_{1}^{r_{1}} p_{2}^{r_{2}} \ldots p_{l}^{r_{l}}$ be the minimal polynomial of $A\in M_{n}(\mathbb{F})$, 
where $p_i\in \FF[x]$ are distinct monic irreducible polynomials and $r_i\in \mathbb{N}$.
Let  $V_{i} = \ker(p_{i}^{r_{i}}(A)), \ i=1, \ldots, l$. Then,

\begin{enumerate}
\item[(i)] $V=V_{1}\oplus \cdots \oplus V_l$,
\item[(ii)] $V_i$ is invariant for $A$,
\item[(iii)] the minimal polynomial of $A_i=A_{|V_i}$ is $p_i^{r_i}$.
\end{enumerate}
\end{theorem}

The centralizer of $A$ over $\mathbb{F}$ is the algebra
$Z(A)=\{ B\in M_{n}(\mathbb{F}): AB = BA\}$. \\
The role of the centralizer is key to analyze important algebraic properties of the endomorphism (\cite{Supru68}). 

 \medskip

The next proposition allows us to reduce the study of the centralizer to the case where the minimal polynomial is of the form $m_{A}=p^{r}$, with $p\in\mathbb{F}[x]$ irreducible.

\begin{proposition} \cite{FillHL77}\label{hinvdecomp}
Let $A$ and $B$ be endomorphisms on finite dimensional vector spaces $V$ and $W$, respectively, over a field $\FF$. The following properties are equivalent:
\begin{enumerate}
\item The minimal polynomials of $A$ and $B$ are relatively prime.
\item
$Z(A \oplus B) = Z(A) \oplus Z(B)$.

 \end{enumerate}
\end{proposition}

From now on we will assume that the characteristic polynomial of $A$ is $p_A=p^r$ with $p=x^s+c_{s-1}x^{s-1}+ \ldots +c_1x+c_0$ irreducible.
We denote by $C$ the companion matrix of  $p$
\begin{equation} \label{companion}
 C=\left[\begin{array}{ccccc}
0 & 0 &\ldots & 0 & -c_0 \\
1 & 0 & \ldots & 0 & -c_1 \\
0 & 1 &\ldots & 0 & -c_2 \\
 \vdots &\vdots & \ddots  & \vdots   & \vdots  \\
0 & 0& \ldots & 1 & -c_{s-1}
\end{array}\right]
\in M_{s}(\mathbb{F}).
\end{equation}

\medskip

Knowing the centralizer of a matrix, we can obtain the centralizer of any other similar one. In order to obtain them, 
it is convenient to describe the centralizer of a canonical form. 

\medskip

One of the most useful canonical forms for the similarity of endomorphisms over a finite dimensional space is the Jordan canonical 
form. It allows us to easily know the determinant, characteristic and minimal polynomials, eigenvalues and eigenvectors and rank of 
the endomorphism, among others. We recall next two generalizations of it over arbitrary and perfect fields, respectively.

\subsection{The generalized Jordan canonical form}

The primary rational canonical form of a matrix over a field under similarity is a generalization of the Jordan canonical form. The name comes 
from the fact that it can be obtained using the operations of a field (rational operations) (see \cite{Ro64}).

\begin{theorem}[Primary rational canonical form or generalized Jordan canonical form, \cite{Ro64,Per70,Ja75,Holtz00}]
Let $p_A=p^r$ with $p=x^s+c_{s-1}x^{s-1}+ \ldots +c_1x+c_0\in\mathbb{F}[x]$ irreducible be the characteristic polynomial 
of $A\in M_{n}(\FF)$. Then, $A$ is similar to
\begin{equation} \label{gJordan}
 G=\diag(G_{1}, G_{2}, \ldots, G_{m}),
 \end{equation}
where
\begin{equation} \label{gJordanblock}
G_{i}=\left[\begin{array}{cccc}
C & 0 & \ldots & 0  \\
E & C & \ldots & 0 \\
 \vdots & \ddots & \ddots  & \vdots  \\
0 & \ldots & E & C
\end{array}\right]
\in M_{s\alpha_{i}}(\mathbb{F}), \quad i=1, \ldots, m,
\end{equation}
$C$ is the companion matrix (\ref{companion}) of $p$, $E$ is the matrix
\begin{equation} \label{E}
E=\left[\begin{array}{cccc}
0 & \ldots & 0 & 1 \\
0 & \ldots & 0 & 0 \\
 \vdots & \ddots &   \vdots &  \vdots  \\
0 & \ldots & 0 & 0 \\
\end{array}\right]
\in M_{s}(\mathbb{F}),
\end{equation}
$\alpha_1\geq \alpha_2\geq \ldots \geq \alpha_m \geq 0$ are integers such that 
$p^{\alpha_i}, \ i=1, \ldots, m$ are the elementary divisors of $G$ and $\sum_{i=1}^{m}\alpha_{i}=r$.

\end{theorem}
\medskip

The following remarks aim at summarizing some properties of the generalized Jordan canonical form.

\begin{remark}\label{rem:varios}
\begin{enumerate}
\item
The matrix $G$ in (\ref{gJordan}) can be found in many references in the literature receiving different 
names: ``rational canonical set'' (\cite{Per70}, here the blocks $G_i$ are called ``hypercompanion matrices''),   ``classical canonical form'' (\cite{Ja75}),  ``Jordan normal form for the field $\FF$'' (\cite{Holtz00},  in this paper it has been obtained by a duality method).

We call the matrix $G$ the {\it generalized Jordan form of} $A$ and $\alpha=(\alpha_1, \alpha_2, \ldots, \alpha_m)$ the 
{\it generalized Segre characteristic of} $A$. 
We will refer to a block $G_i$ as a {\it generalized Jordan block}.
Here, each $\alpha_i$ denotes the number of diagonal blocks in the matrix $G_i$. 

When $\deg(p)=1$, the resulting matrix is the {\it Jordan canonical form} (\cite{Jor1870}).

\item The canonical form (\ref{gJordan}) allows the following decomposition
\begin{equation} \label{DN}
G=D+N=\diag(D_1,\ldots, D_m)+\diag(N_1,\ldots, N_m)
\end{equation}
\begin{equation*}
D_i=\left[\begin{array}{cccc}
C & 0 & \ldots & 0  \\
0 & C & \ldots & 0 \\
 \vdots & \ddots & \ddots  & \vdots  \\
0 & \ldots & 0 & C
\end{array}\right], \quad
N_i=
\left[\begin{array}{cccc}
0 & 0 & \ldots & 0  \\
E & 0 & \ldots & 0 \\
 \vdots & \ddots & \ddots  & \vdots  \\
0 & \ldots & E & 0
\end{array}\right]
\end{equation*}
with $N_{i},D_{i}\in M_{s\alpha_{i}}(\FF)$. In general $DN\neq ND$.

\end{enumerate}
\end{remark}

\medskip
\begin{remark}\label{rem:baseJ}

A generalized Jordan basis can be written as ${\mathcal B}=\{v^{(1)}, v^{(2)},\ldots,v^{(m)}\}$ where $v^{(i)}$ is a
{\it generalized Jordan chain.} Each one of them contains $\alpha_i$ {\it partial chains}, that is

$$v^{(1)}=\{\underbrace{w_{1,1}, \ldots, w_{1,s}}_{v^{(1)}_1}, \underbrace{w_{1, s+1}, \ldots, w_{1,2s}}_{v^{(1)}_2},\ldots, \underbrace{w_{1,(\alpha_{1}-1)s+1}, \ldots, w_{1,\alpha_1s}}_{v^{(1)}_{\alpha_1}}\}$$
$$v^{(2)}=\{\underbrace{w_{2,1}, \ldots, w_{2,s}}_{v^{(2)}_1},\ldots, \underbrace{w_{2,(\alpha_{2}-1)s+1}, \ldots, w_{2,\alpha_{2}s}}_{v^{(2)}_{\alpha_2}}\}$$
$$\ldots$$
$$v^{(m)}=\{\underbrace{w_{m,1}, \ldots, w_{m,s}}_{v^{(m)}_1},\ldots, \underbrace{w_{m,(\alpha_{m}-1)s+1}, \ldots, w_{m,\alpha_ms}}_{v^{(m)}_{\alpha_m}}\}$$
such that for $i=1,\ldots,m,$
$$\begin{array}{l}
w_{i,1}\in\ker (p^{\alpha_{i}}(G))\setminus \ker (p^{\alpha_{i}-1}(G)),\\
Gw_{i,j}=w_{i,j+1}, \  j=1, \ldots, s\alpha_i, \quad j\neq ks, \quad k=1, \ldots, \alpha_i,  \\
w_{i,ks+1}= p^{k}(G)w_{i,1}, \quad k=1\ldots, \alpha_i-1. \\
\end{array}$$

\end{remark}
\medskip

\begin{example}\label{example:Jordan}

Let $\alpha=(3,2)$, that is, $G=\diag(G_1, G_2)$ with
$$G_1=\left[\begin{array}{ccc}
C & 0 & 0\\
E & C & 0\\
0& E & C
\end{array}\right]\in M_{3s}(\FF), \quad 
G_2=\left[\begin{array}{cc}
C & 0 \\
E & C 
\end{array}\right]\in M_{2s}(\FF). $$
In this case the minimal polynomial of $G$ is $m_G=p^3$ and $\deg(p)=s$. 
Let ${\mathcal B}=\{v^{(1)}, v^{(2)}\}$ be the generalized Jordan basis. Each Jordan chain $v^{(i)}$
contains $\alpha_i$ partial chains

$$v^{(1)}=\{\underbrace{w_{1,1}, \ldots, w_{1,s}}_{v^{(1)}_1}, \underbrace{w_{1, s+1}, \ldots, w_{1,2s}}_{v^{(1)}_2}, \underbrace{w_{1,2s+1}, \ldots, w_{1,3s}}_{v^{(1)}_{3}}\}$$
$$v^{(2)}=\{\underbrace{w_{2,1}, \ldots, w_{2,s}}_{v^{(2)}_1}, \underbrace{w_{2,s+1}, \ldots, w_{2,2s}}_{v^{(2)}_{2}}\}$$

\end{example}

\subsection{The generalized Jordan canonical form of the first kind}

Concerning the existence of canonical forms of matrices for the similarity equivalence relation, particular attention deserves the case when the polynomial $p$ is separable. In this case, another canonical form can be obtained which allows the so called Jordan-Chevalley decomposition of a matrix (\cite{Per70}).  We recall here the results.

The existence of the Jordan-Chevalley decomposition makes easier the study of certain properties of the endomorphism. For instance, one example is the study of the lattices of its hyperinvariant and characteristic subspaces (see \cite{BrickFill65,MMR18}). In particular, it makes easier the obtention of the centralizer of the endomorphism, as we will see later. 

\begin{theorem}[Generalized Jordan canonical form of the first kind, \cite{HoffKun71,Ro70}] \label{jordan_separable}
Let $p_A=p^r$ with $p=x^s+c_{s-1}x^{s-1}+ \ldots +c_1x+c_0\in\mathbb{F}[x]$ irreducible and separable be the characteristic polynomial 
of $A\in M_{n}(\FF)$. Then, $A$ is similar to a matrix of the form
\begin{equation} \label{gJordan1type}
 G=\diag(G_{1}, G_{2}, \ldots, G_{m}),
 \end{equation}
where
\begin{equation*}
G_{i}=\left[\begin{array}{cccc}
C & 0 & \ldots & 0  \\
I & C & \ldots & 0 \\
 \vdots & \ddots & \ddots  & \vdots  \\
0 & \ldots & I & C
\end{array}\right]
\in M_{s\alpha_{i}}(\mathbb{F}), \quad i=1, \ldots, m,
\end{equation*}
$C$ is the companion matrix (\ref{companion}) of $p$, $I$ is the identity matrix and 
$\alpha_1\geq \alpha_2\geq \ldots \geq \alpha_m\geq 0$ are integers such that 
$p^{\alpha_i}, \ i=1, \ldots, m$ are the elementary divisors of $G$ and $\sum_{i=1}^{m}\alpha_{i}=r$.
\end{theorem}

\begin{remark}\label{rem:varios1type}

\begin{enumerate}
\item This canonical form is known as the generalized Jordan form of {\it the first kind}~(\cite{Dalalyan2014}). A particular case 
of this canonical form is the real Jordan canonical form (\cite{Goh}).
When $\deg(p)=1$, this form also reduces to the Jordan canonical form.

\item When $p$ is separable, the matrix (\ref{gJordan1type}) is obviously similar to the generalized Jordan form  (\ref{gJordan})  
(for a proof see \cite{Ro70}).

\item The canonical form (\ref{gJordan1type}) allows a decomposition analogous to  (\ref{DN}), which is the following 

\begin{equation*}
G=D+N=\diag(D_1,\ldots, D_m)+\diag(N_1,\ldots, N_m)
\end{equation*}
\begin{equation*}
D_i=\left[\begin{array}{cccc}
C & 0 & \ldots & 0  \\
0 & C & \ldots & 0 \\
 \vdots & \ddots & \ddots  & \vdots  \\
0 & \ldots & 0 & C
\end{array}\right], \quad
N_i=
\left[\begin{array}{cccc}
0 & 0 & \ldots & 0  \\
I & 0 & \ldots & 0 \\
 \vdots & \ddots & \ddots  & \vdots  \\
0 & \ldots & I & 0
\end{array}\right]
\end{equation*}
with $N_{i},D_{i}\in M_{s\alpha_{i}}(\FF)$. Now, $DN=ND$, and the decomposition is known as {\it Jordan-Chevalley decomposition}. 
In fact, given $p_A=p^r$, $p$ is separable if and only if $A$ admits Jordan-Chevalley decomposition. See, for instance,~\cite{Ro70}. 
\end{enumerate}
\end{remark}

\section{The generalized Weyr canonical form}\label{sec:Weyr}

A canonical form of an endomorphism under similarity, relevant to theoretical and applied mathematics, is the Weyr canonical form. It has been obtained when the minimal polynomial splits over $\mathbb{F}$ (therefore, it exists over algebraically closed fields). See  \cite{Omeara} for details. 
This section is devoted to obtain a generalization  of the Weyr form over an arbitrary field, which will be called the {\it  generalized Weyr canonical form} (or {\it Weyr primary rational canonical form}).

One important feature leading the present work is that the Weyr canonical form allows to describe the matrices in the centralizer of  an endomorphism  in an upper triangular form (see~\cite{Omeara}).
We will generalize it to arbitrary fields, and will obtain the corresponding upper triangular form of the matrices in the centralizer. We will use this property  to calculate the determinant of the elements of  the centralizer. 

\medskip

The Weyr canonical form can be obtained from the Jordan canonical form reordering appropriately the vectors of a 
Jordan basis. In fact, it is associated to the conjugate partition of the Segre characteristic of the endomorphism (see \cite{Omeara}).
To obtain the generalized Weyr canonical form we use the same sort of transformation,
as we see next.

\medskip

According to Proposition~\ref{hinvdecomp}, we will assume that the minimal polynomial of the matrix $A\in M_n(\FF)$ 
is of the form $m_{A}=p^{r}$, where $p$ is irreducible over $\mathbb{F}$. The generalized Weyr canonical form will also be associated to
the conjugate partition of the generalized Segre characteristic of $A$.

\medskip

Let $\alpha=(\alpha_1, \ldots,\alpha _m)$ be the generalized Segre characteristic of $A$, and ${\mathcal B}=\{v^{(1)}, v^{(2)},\ldots,v^{(m)}\}$
the generalized Jordan basis defined in Remark~\ref{rem:baseJ},
where $v^{(i)}=\{v^{(i)}_{1},\ldots,v^{(i)}_{\alpha_{i}}\}$ and each partial chain $v^{(i)}_{j}$ is composed by a collection of $s$ vectors.
To obtain a generalized Weyr basis from it, we proceed analogously to the obtention of the Weyr basis, but replacing vectors by partial chains. The relations among the partial chains can be sketched as follows 

$$\begin{array}{cccc}
  v^{(1)}_{\alpha_1} & \leftarrow\cdots\leftarrow & v^{(1)}_{1} &  \\
 v^{(2)}_{\alpha_2} &\leftarrow\cdots\leftarrow &v^{(2)}_{1} & \\
\vdots && \vdots &\\
 v^{(m)}_{\alpha_m} & \leftarrow\cdots \leftarrow & v^{(m)}_{1} & 
\end{array}$$
In order to renumber the partial chains according to its absolute position in the basis we need to introduce some notation. 
Let $(\beta_1,\ldots,\beta_h)$ be the different values of the generalized Segre partition and $(n_1,\ldots,n_h)$ its 
frequencies. Let $(\mu_1,\ldots,\mu_h)$ be the cumulative frequencies of $\beta_{i}$ ($\mu_i= \mu_{i-1} + n_i$). 
For $i=1,\ldots,m$, let $\sigma_{i}=\alpha_{1}+\ldots+\alpha_{i}$ 
$(\sigma_{0}=0)$. For $\sigma_{i-1}<j\leq \sigma_{i}$, we define 
$$v_{j}=v_{j-\sigma_{i-1}}^{(i)}.$$ 
Then, the partial chains of the basis can be described as follows

$$ \begin{array}{lccc}
   \left. \begin{array}{ccccccccccc}
 & v_{\sigma_1} & \cdots &&  v_{\sigma_1-\beta_h+1}  & \cdots &   v_{\sigma_1-\beta_{k}+1}  & \cdots & v_{\sigma_1-\beta_2+1}  &\cdots &  v_{\sigma_1-\beta_1+1}  \\
  &\vdots&&&\vdots &&\vdots &&\vdots&&\vdots\\
  &v_{\sigma_{\mu_{1}}} &\cdots &&v_{\sigma_{\mu_{1}}-\beta_h+1} &\cdots &v_{\sigma_{\mu_{1}}-\beta_{k}+1} & \cdots  & v_{\sigma_{\mu_{1}}-\beta_2+1}& \cdots & v_{\sigma_{\mu_{1}}-\beta_1+1} \end{array} \right\}&n_{1}\\
 \left. \begin{array}{cccccccccccccccccc}&&\vdots&   & &  &&& & \vdots &  & &&&&& &\vdots  \end{array}\right. &\vdots  \\
  \left.\begin{array}{ccccccccc}
v_{\sigma_{\mu_{k}-n_{k}+1}}& \cdots &v_{\sigma_{\mu_{k}-n_{k}}-\beta_h+2} & \cdots & v_{\sigma_{\mu_{k}-n_{k}}-\beta_{k}+2}&               &           &          &\\
\vdots&   &\vdots &          &        \vdots   & & & & \\ 
v_{\sigma_{\mu_{k}}}& \cdots &v_{\sigma_{\mu_{k}}-\beta_h+1} & \cdots & v_{\sigma_{\mu_{k}}-\beta_{k}+1} &          &              &            &\end{array}\right \}&n_{k}\\
 \left. \begin{array}{ccccccccccc}&&\vdots&   & &  &&& & \vdots  \end{array}\right.&\vdots\\
\left.\begin{array}{ccccccccc}
v_{\sigma_{\mu_{h}-n_{h}+1}} & \cdots  & v_{\sigma_{\mu_{h}-n_{h}+1}-\beta_h+2}&&&&&&\\
 \vdots &&\vdots &&&&&& \\ 
v_{\sigma_{\mu_{h}}} & \cdots  & v_{\sigma_{\mu_{h}}-\beta_h+1}&&&&&& 
\end{array} \right \}&n_{h}\end{array}$$
Now, taking this basis in vertical order we obtain the generalized Weyr form.
In more detail, if we write the identity matrix as 
$$I_{n}=\left[\begin{array}{cccccc}
I_{*(1)}, \ldots, I_{*(\sigma_1)}, I_{*(\sigma_1+1)}, \ldots, I_{*(\sigma_{\mu_{h}})}
\end{array}\right],$$
where $I_{*(j)}$ denotes a block of $s$ consecutive columns of $I_{n}$ (notice that $\sigma_{\mu_{h}}s=n$), the permutation 
matrix reordering the basis is 
\begin{eqnarray}\label{permutationmatrix} P=\left[
I_{*(\sigma_1)} \ \ldots \ I_{*(\sigma_{\mu_{h}})} \ | \
I_{*(\sigma_1-1)} \ \ldots \ I_{*(\sigma_{\mu_{h}}-1)} \ | \ \cdots  \ | \
I_{*(\sigma_1-\beta_{h}+1)}  \ \ldots \  I_{*(\sigma_{\mu_{h}}-\beta_{h}+1)} \ | \
\right.\nonumber \\
\left.
\hspace{1.25cm}\ldots\ | \ I_{*(\sigma_1-\beta_k)}  \ \ldots \  I_{*(\sigma_{\mu_{k-1}}-\beta_k)} \ | \ \cdots  \ | \ I_{*(\sigma_1-\beta_{k-1}+1)}  \ \ldots \  I_{*(\sigma_{\mu_{k-1}}-\beta_{k-1}+1)} \ | 
\right.\nonumber\\
\left.
\hspace{3cm} \cdots \  |  \ I_{*(\sigma_1-\beta_2)}  \ \ldots  \ I_{*(\sigma_{\mu_{1}}-\beta_2)} \ | \ldots | \ I_{*(\sigma_1-\beta_1+1)} \  \ \ldots  \  \ I_{*(\sigma_{\mu_{1}}-\beta_{1}+1)}
\right] \end{eqnarray}
and the following theorem is obtained.

\begin{theorem}\label{theo:forma de Weyr}
Let $G=\diag(G_1, \ldots, G_m)$ be a generalized Jordan matrix as in (\ref{gJordan}). Let $p_{G}=p^{r}$ be its characteristic polynomial with $p$ irreducible and $\deg{(p)}=s$. 
Let $\alpha=(\alpha_1, \ldots, \alpha_m)$ be the generalized Segre characteristic of $G$ and $\tau=(\tau_1, \ldots, \tau_{\alpha_{1}})$ the conjugate partition of $\alpha$.

Then, $G$ is similar to a matrix
\begin{equation}\label{weyrform}
W=\left[\begin{array}{cccccc}
\vspace{-0.15cm}
W_1&E_2& \ldots & 0 &0\\
0&W_2&\ddots& 0 &0\\
\vdots & \vdots & \ddots & \ddots & \vdots \\
0 & 0 & \ldots &W_{\alpha_{1}-1}&E_{\alpha_{1}}\\
0 & 0 & \ldots & 0 &W_{\alpha_{1}}
\end{array}\right],\end{equation}
where

\begin{equation*}
W_i=\left[\begin{array}{cccc}
C&0& \ldots&0\\
0&C&\ldots& 0\\
\vdots & \vdots & \ddots & \vdots\\
0&0 & \ldots & C\\
\end{array}\right]\in M_{s\tau_i}(\mathbb{F}),\quad
i=1, \ldots \alpha_{1},
 \end{equation*}
\begin{equation*}
E_{i+1}=\left[\begin{array}{cccc}
E&0&\ldots&0\\
0&E&\ldots&0\\
\vdots& \vdots & \ddots& \vdots\\
0 & 0 &\ldots &E\\
0 & 0 &\ldots&0\\
\vdots&\vdots&\vdots&\vdots\\
0 & 0 &\ldots&0\\
\end{array}\right]\in M_{s\tau_i\times s\tau_{i+1}}(\mathbb{F}),\quad i=1, \ldots \alpha_{1}-1, \end{equation*}	
with $E$ defined as in~(\ref{E}).
\end{theorem}

\begin{proof}
Representing $G$ with respect to the basis reordered according to the permutation matrix $P$ described 
in~(\ref{permutationmatrix}), we  obtain the desired result, i.e. $P^{-1}GP=W$. 
Observe that, because of the reordering chosen, the sizes of the resulting diagonal blocks are given by $\tau$, the conjugate partition of $\alpha$.

\end{proof}

\begin{example}

Let $G$ be a generalized Jordan form with generalized Segre characteristic $\alpha=(3,2,2)$ and $${\mathcal B}=\{v^{(1)},v^{(2)},v^{(3)}\}=\{\{v^{(1)}_{1},v^{(1)}_{2},v^{(1)}_{3}\},\{v^{(2)}_{1},v^{(2)}_{2}\},
\{v^{(3)}_{1},v^{(3)}_{2}\}\},$$ the corresponding basis. We can sketch the relations among the partial chains of the basis as  

$$\begin{array}{ccc}
  v^{(1)}_{3} & \leftarrow v^{(1)}_{2} &\leftarrow v^{(1)}_{1}   \\
 
 v^{(2)}_{2} &\leftarrow v^{(2)}_{1} &  \\

 v^{(3)}_{2} & \leftarrow v^{(3)}_{1} & 
\end{array}$$
Notice that $(\beta_{1},\beta_{2})=(3,2),$ $(n_{1},n_{2})=(1,2)$, $(\mu_{1},\mu_{2})=(1,3)$, $(\sigma_{1},\sigma_{2},\sigma_{3})=(3,5,7)$.
For $\sigma_{i-1}<j\leq \sigma_{i}$, we define 
$$v_{j}=v_{j-\sigma_{i-1}}^{(i)},$$ 
then, the partial chains of the basis can be described as follows
$$ \begin{array}{lcc}
   \left. \begin{array}{cccc}
  v_{3} &  v_{2}  &  v_{1}\end{array}   \right\}&n_{1}=1\\
 
  \left.\begin{array}{ccc}
v_{5}&v_{4}& \\
 
v_{7}&v_{6} \end{array}\quad\right \}&n_{2}=2
\end{array}$$
Now, if we take this basis in vertical order 
${\mathcal B}'=\{\{v_3, v_5,v_{7}\}, \{v_{2},v_{4},v_{6}\}, \{v_{1}\}\},$ we obtain the generalized Weyr basis associated to the conjugate 
partition of $\alpha$, $\tau=(3,3,1)$.
The permutation matrix which reorders the basis is 
$$P=\left[\begin{array}{ccccccc}
I_{*(3)} & I_{*(5)}& I_{*(7)}&I_{*(2)} &  I_{*(4)}&I_{*(6)}& I_{*(1)}\end{array}\right],$$
and the resulting matrix is 
$$W=\left[\begin{array}{ccc|ccc|c}
C & &&E & &&\\
 & C & && E && \\
 && C &&& E&\\
 \hline
&  & &C  & &&E \\
& & & &C &&  \\
&&&& &C & \\
\hline
& & & & && C
\end{array}\right]\in M_{7s}(\FF).$$

\end{example}

\begin{remark} The  Weyr characteristic  can be obtained in terms of the kernels of $p^{i}(W)$ (\cite{Omeara}).
Analogously, the generalized Weyr characteristic $\tau$ can be computed as

$$\begin{array}{l}
\tau_{1}= \frac{1}{s}\dim(\ker(p(W)),\\
\\
\tau_{2} = \frac{1}{s}(\dim(\ker(p^2(W))-\dim(\ker(p(W))),\\
          \vdots\\
\tau_{\alpha_1}=\frac{1}{s}(\dim(\ker(p^{\alpha_{1}}(W)) -\dim(\ker(p^{\alpha_{1}-1}(W))). 
\end{array}$$
\end{remark}

\section{The centralizer of a matrix over an arbitrary field}\label{gJm}

The centralizer $Z(A)$ of a matrix $A\in M_{n}(\FF)$ is known when the characteristic polynomial splits over $\FF$ (see~\cite{Supru68}), for nonderogatory matrices (see~\cite{Calde1}) and for $\FF=\mathbb{R}$ (see~\cite{Goh}). 
\medskip

In this section we obtain the centralizer of a matrix in the generalized Jordan form  (\ref{gJordan}), therefore for arbitrary fields.
We can obtain it thanks to the structure of the generalized Jordan form and its behavior face to the matrix multiplication. We achieve the result in three steps: in Subsection \ref{sec:Companion} we recall the centralizer of a companion matrix (i.e., nonderogatory matrix), in 
Subsection \ref{centralizer-Jblock} we obtain the centralizer of a generalized Jordan block, and finally in 
Subsection \ref{The centralizer} we find the centralizer of a generalized Jordan matrix.
To prove our results, we introduce in Subsection \ref{centralizer-Jblock}  some technical 
lemmas concerning the obtention of solutions of certain matrix equations.

\subsection{The centralizer of a companion matrix}\label{sec:Companion}

Let $C\in M_{n}(\mathbb{F})$ be the companion matrix  (\ref{companion}) and $m_C(x)=x^{n}+c_{n-1}x^{n-1}+\ldots +c_1x+c_{0}$ its minimal polynomial.

\medskip

In the following lemma we recall the characterization of the centralizer of a companion matrix (\cite{Calde1}). Next, we give a 
more simple proof than that of~ \cite{Calde1}. Moreover, the technique we use to prove it, is also used later to obtain the 
centralizer of the generalized Jordan form.

\begin{lemma}\label{lemmacentracom}~\cite{Calde1}
The centralizer $Z(C)$ of the companion matrix $C$ is 
\[\{X\in M_{n}(\mathbb{F}): X=\left[\begin{array}{cccc}
v & Cv & \ldots & C^{n-1}v
\end{array}\right] , \quad v\in \FF^n\}.\]
\end{lemma}

\begin{proof}
From the definition of the centralizer, we have that
\[X\in Z(C) \Leftrightarrow CX=XC.\]
As, 
\[CX=C \ \left[\begin{array}{ccc} X_{\ast 1} & \ldots & X_{\ast n}\end{array}\right] =\left[\begin{array}{ccc} CX_{\ast 1} & \ldots & CX_{\ast n}\end{array}\right],\]
\[XC=\left[\begin{array}{cccc} X_{\ast 2} & \ldots & X_{\ast n} & -c_0X_{\ast 1}- \ldots -c_{n-1}X_{\ast n}\end{array}\right],\]
identifying columns we obtain
\[ \left .\begin{array}{l}
X_{\ast 2}=CX_{\ast 1} \\
X_{\ast 3}=CX_{\ast 2}=C^2X_{\ast 1} \\
\ldots \\
X_{\ast n}=CX_{\ast n-1}=C^{n-1}X_{\ast 1}
\end{array}\right\}\]
therefore
\[CX=XC \Leftrightarrow X=\left[\begin{array}{cccc}
X_{\ast 1} & CX_{\ast 1} & \ldots & C^{n-1}X_{\ast 1}
\end{array}\right].\]
\end{proof}

As a consequence of the above lemma we see that the matrices in the centralizer $Z(C)$ can be parametrized in terms of the elements
of the first column. Next corollary shows another parametrization in terms of the last row.  

\begin{corollary}\label{corolcentracom}
Let $X=[x_{i,j}]\in Z(C)$. Then,
\begin{equation} \label{aux}
x_{n-i,1}=c_{n-i}x_{n,1}+c_{n-i+1}x_{n,2}+\ldots +c_{n-1}x_{n,i}+ x_{n,i+1}, \quad i=1, \ldots, n-1.
\end{equation}
\end{corollary}

\begin{proof}
As
 \[X=
 \left[\begin{array}{cccc}
x_{1,1} & x_{1,2} & \ldots  & x_{1,n} \\
x_{2,1} & x_{2,2}  & \ldots  & x_{2,n} \\
& & \ddots & \\
x_{n,1} & x_{n,2}  & \ldots  & x_{n,n} \\
 \end{array}\right]=
 \left[\begin{array}{cccc}
X_{\ast1} & CX_{\ast1}  & \ldots  & C^{n-1}X_{\ast1}
 \end{array}\right],\]
for $j=1, \ldots, n-1$,
\[ \left[\begin{array}{ccccc}
x_{1,j+1} \\ x_{2,j+1} \\ \ldots  \\ x_{n-1,j+1} \\ x_{n,j+1}
 \end{array}\right]=
C \left[\begin{array}{ccccc}
x_{1,j} \\ x_{2,j} \\ \ldots  \\ x_{n-1,j} \\ x_{n,j}
 \end{array}\right]=
 \left[\begin{array}{ccccc}
0 & 0 & \ldots & 0 & -c_0 \\
1 & 0 & \ldots & 0 & -c_1 \\
  & & \ddots &   &      \\
0 & 0 & \ldots & 0 & -c_{n-2} \\
0 & 0 & \ldots & 1 & -c_{n-1}
 \end{array}\right]
 \left[\begin{array}{ccccc}
x_{1,j} \\ x_{2,j} \\ \ldots  \\ x_{n-1,j} \\ x_{n,j}
 \end{array}\right]=
 \left[\begin{array}{c}
-c_0x_{n,j} \\
x_{1,j}-c_1x_{n,j} \\
\ldots \\
x_{n-2,j}-c_{n-2}x_{n,j} \\
x_{n-1,j}-c_{n-1}x_{n,j} \\
 \end{array}\right].\]
It means that for $i=1, \ldots, n, \quad j=1, \ldots, n-1 \quad (x_{0,j}=0)$
\begin{equation} \label{cond3b}
x_{n-i,j}=c_{n-i}x_{n,j}+x_{n-i+1,j+1}.
\end{equation}
Applying (\ref{cond3b}) repeatedly we obtain the conclusion:
\[ x_{n-i,1}=c_{n-i}x_{n,1}+x_{n-i+1,2}=\]
\[=c_{n-i}x_{n,1}+c_{n-i+1}x_{n,2}+ x_{n-i+2,3}=\]
\[\ldots\]
\[=c_{n-i}x_{n,1}+c_{n-i+1}x_{n,2}+\ldots +c_{n-1}x_{n,i}+ x_{n,i+1}.\]
\end{proof}

As a consequence, the following results arise.

\begin{corollary}\label{corodim}\cite{Calde2}
Assume that $C\in M_{n}(\mathbb{F})$ is the companion matrix  (\ref{companion}). Then,  
it is satisfied that
 \begin{enumerate}
 \item $\dim (Z(C))=n.$
  \item Let $X\in Z(C)$, then
  $$ \det(X)=0  \Leftrightarrow X_{\ast 1} = 0\Leftrightarrow X_{n\ast}=0.$$
  \item If the polynomial associated to $C$ is irreducible, then
  $$ \det(X)=0 \Leftrightarrow  X=0.$$
  
\end{enumerate}
\end{corollary}

\medskip

\subsection{The centralizer of a generalized Jordan block}\label{centralizer-Jblock}

Let $G$ be a generalized Jordan block as in (\ref{gJordanblock}). Our target now is to obtain the centralizer of $G$.

We introduce the following notation: given a matrix $X=[x_{i,j}]_{1\leq i,j\leq n}\in M_{n}(\mathbb{F})$, we denote by $\tilde{X}\in M_{n}(\mathbb{F})$
 \begin{equation}\label{deftildeX}
 \tilde{X}=\left[\begin{array}{ccccc}
0 & x_{n,1} & x_{n,2} & \ldots  & x_{n,n-1} \\
0 &     0     & x_{n,1} & \ldots  & x_{n,n-2} \\
& & & \ddots & \\
0 & 0 & 0 & \ldots  & x_{n,1} \\
0 & 0 & 0 & \ldots  & 0 \\
 \end{array}\right].\end{equation}

The next lemma is key to compute the centralizer of a generalized Jordan matrix $G$.

\begin{lemma}\label{lemaequationsxtilde}
Let $C\in M_{n}(\mathbb{F})$ be a companion matrix as in (\ref{companion})  and
$E$ a matrix as in (\ref{E}).
Let $X,Y\in Z(C)$ and $T\in M_{n}(\mathbb{F})$. Then
\begin{equation*}\label{eqnN2}
        \left.\begin{array}{c}          
           EX+CT=TC+YE\\
           \end{array}\right.\Leftrightarrow \ X=Y,\quad T=T'+\tilde{X},
\end{equation*}
where $T'\in Z(C)$ and $\tilde{X}$ is defined as in (\ref{deftildeX}).
\end{lemma}

\begin{proof}
By Lemma~\ref{lemmacentracom} we have that
\[X\in Z(C) \Leftrightarrow   X= \left[\begin{array}{cccc} X_{\ast 1} & CX_{\ast 1} & \ldots & C^{n-1}X_{\ast 1} \end{array}\right],\]
\[Y \in Z(C) \Leftrightarrow  Y= \left[\begin{array}{cccc} Y_{\ast 1} & CY_{\ast 1} & \ldots & C^{n-1}Y_{\ast 1}  \end{array}\right].\]
Observe that
\[
EX= \left[\begin{array}{cccc}
x_{n,1} & \ldots & x_{n,n-1} & x_{n,n}
 \\
0 & \ldots & 0 & 0 \\
  & \ddots &   &   \\
0 & \ldots & 0 & 0 \\
 \end{array}\right], \quad
 YE= \left[\begin{array}{cccc}
0 & \ldots & 0 & y_{1,1}
 \\
0 & \ldots & 0 & y_{2,1} \\
  & \ddots &   &   \\
0 & \ldots & 0 & y_{n,1} \\
 \end{array}\right].\]
If we denote by $e_{i}$ the i-th vector of the canonical basis of $\mathbb{F}^n$, then
\[\begin{array}{ll}
EX+CT=& \left[\begin{array}{ccc} x_{n,1}e_1 & \ldots & x_{n,n}e_1 \end{array}\right]+  \left[\begin{array}{ccc} CT_{\ast 1} & \ldots & CT_{\ast n} \end{array}\right],\\
TC+YE=&  \left[\begin{array}{cccc} T_{\ast 2} & \ldots & T_{\ast n} & -c_0T_{\ast 1}- \ldots -c_{n-1}T_{\ast n}+Y_{\ast 1} \end{array}\right].\end{array}\]
Identifying the two expressions we obtain
\begin{equation} \label{X}
 \left.\begin{array}{l}
T_{\ast 2}=x_{n,1}e_1+CT_{\ast 1} \\
T_{\ast 3}=x_{n,2}e_1+CT_{\ast 2}=x_{n,2}e_1+C(x_{n,1}e_1+CT_{\ast 1})=(x_{n,2} +x_{n,1}C)e_1+C^2T_{\ast 1} \\
\ldots \\
T_{\ast n}=x_{n,n-1}e_1+CT_{\ast n-1}=(x_{n,n-1}+x_{n,n-2}C +\ldots +x_{n,1}C^{n-2})e_1+C^{n-1}T_{\ast 1} \\
-c_0T_{\ast 1}- \ldots -c_{n-1}T_{\ast n}+Y_{\ast 1}=x_{n,n}e_1+CT_{\ast n}
\end{array}\right\}
\end{equation}
Hence
\[
 \left.\begin{array}{l}
T_{\ast 2}=x_{n,1}e_1+CT_{\ast 1} \\
T_{\ast 3}=x_{n,2}e_1 +x_{n1}e_2+C^2T_{\ast 1} \\
\dots \\
T_{\ast n}=x_{n,n-1}e_1+x_{n,n-2}e_2 +\ldots +x_{n,1}e_{n-1}+C^{n-1}T_{\ast 1} \\
-c_0T_{\ast 1}- \ldots -c_{n-1}T_{\ast n}+Y_{\ast 1}=x_{n,n}e_1+CT_{\ast n}
\end{array}\right\}\]
and replacing $T_{\ast 2}, \ldots, T_{\ast n}$  into the last equation
\[\begin{array}{l}
c_0T_{\ast 1}+\\
c_1(x_{n,1}e_1+CT_{\ast 1})+\\
c_2(x_{n,2}e_1 +x_{n,1}e_2+C^2T_{\ast 1})+ \\
\ldots \\
c_{n-1}(x_{n,n-1}e_1+x_{n,n-2}e_2 +\ldots +x_{n,1}e_{n-1}+C^{n-1}T_{\ast 1}) + \\
x_{n,n}e_1 + \\
+C(x_{n,n-1}e_1+x_{n,n-2}e_2 +\ldots +x_{n,1}e_{n-1}+C^{n-1}T_{\ast 1})=\end{array}\]

\medskip

\[\begin{array}{l}
=(c_0I_{n}+c_1C+ \ldots +c_{n-1}C^{n-1}+C^n)T_{\ast 1}+\\
c_1x_{n,1}e_1+\\
c_2(x_{n,2}e_1 +x_{n,1}e_2)+ \\
\ldots \\
c_{n-2}(x_{n,n-2}e_1+x_{n,n-3}e_2 +\ldots +x_{n,1}e_{n-2}) + \\
c_{n-1}(x_{n,n-1}e_1+x_{n,n-2}e_2 +\ldots +x_{n,2}e_{n-2}+x_{n,1}e_{n-1}) + \\
x_{n,n}e_1
+x_{n,n-1}e_2+x_{n,n-2}e_3 +\ldots +x_{n,1}e_{n}=\end{array}\]

\begin{equation*} \label{condb}
 =\left[\begin{array}{r}
c_1x_{n,1}+c_2x_{n,2}+ \ldots +c_{n-2}x_{n,n-2}+c_{n-1}x_{n,n-1}+x_{n,n} \\
c_2x_{n,1}+c_3x_{n,2}+ \ldots +c_{n-1}x_{n,n-2}+x_{n,n-1}\\
\ldots \\
c_{n-2}x_{n,1}+c_{n-1}x_{n,2}+x_{n,3}\\
c_{n-1}x_{n,1}+x_{n,2}\\
x_{n,1}
 \end{array}\right]=
 Y_{\ast 1},
\end{equation*}

\noindent
Taking into account Corollary \ref{corolcentracom} we conclude that
\[x_{i,1}=y_{i,1}, \quad i=1, \ldots, n,\]
therefore $X_{\ast 1}=Y_{\ast 1}$, and since $X,Y\in Z(C)$ we obtain that $X=Y$. Moreover, from equations (\ref{X}) denoting
$T= \left[\begin{array}{cccc} T_{\ast 1} & T_{\ast 2} & \ldots & T_{\ast n} \end{array}\right]
$, 
\[T= \left[\begin{array}{ccccc}
T_{\ast 1} & CT_{\ast 1} & C^{2}T_{\ast 1} & \ldots & C^{n-1}T_{\ast 1}
 \end{array}\right]
 + \left[\begin{array}{ccccc}
0 & x_{n,1} & x_{n,2} & \ldots  & x_{n,n-1} \\
0 &     0     & x_{n,1} & \ldots  & x_{n,n-2} \\
& & & \ddots & \\
0 & 0 & 0 & \ldots  & x_{n,1} \\
0 & 0 & 0 & \ldots  & 0 \\
 \end{array}\right],\]
i.e. $T=T'+\tilde{X}$ where $T'\in Z(C)$ and $\tilde{X}$ defined as in (\ref{deftildeX})
as desired.

Conversely, assume that $X, T'\in Z(C)$, $Y=X$ and $T=T'+\tilde{X}$ where $\tilde{X}$ is defined as in (\ref{deftildeX}).
To prove that  
\begin{equation*}\label{eqnN2_converse}    
           EX+CT=TC+YE,
\end{equation*}
it is enough to prove that 
\begin{equation*}\label{eqnN3_converse}       
           EX+C\tilde{X}=\tilde{X}C+XE,\\
\end{equation*}
and this equation is satisfied whenever (\ref{aux}) is satisfied. 
But it fulfills because,
$X\in Z(C)$.


\end{proof}

\begin{corollary}\label{coro:dim}
 The set 
 $${\mathcal M}=\{(X,T)\in M_{n}(\FF)\times M_{n}(\FF): X\in Z(C),\ T=T'+\tilde{X},\ T'\in Z(C)\}$$ is a vector subspace of dimension
 $$\dim({\mathcal M})=2n.$$
\end{corollary}

Notice that since $E\tilde{X}=\tilde{X}E=0$, Lemma~\ref{lemaequationsxtilde} can be stated in a more general form as in the following lemma.

\begin{lemma}\label{lemaequations}
Let $C\in M_{n}(\mathbb{F})$ be a companion matrix as in (\ref{companion}) and
$E$ a matrix as in (\ref{E}).
Let  $X',Y'\in Z(C)$, $T, A \in M_{n}(\mathbb{F})$ and $X=X'+\tilde{A}$,  $Y=Y'+\tilde{A}$ with  $\tilde{A}$ defined as
in (\ref{deftildeX}). Then
\begin{equation*}\label{eqn2}
        \left.\begin{array}{c}
           EX+CT=TC+YE\\
           \end{array}\right. \Leftrightarrow\ X=Y,\quad T=T'+\tilde{X},
\end{equation*}
where $T'\in Z(C)$ and $\tilde{X}$ is as in (\ref{deftildeX}).
\end{lemma}

Particular cases of the previous results are stated in the next corollary; they will be used later.

\begin{corollary}\label{coreq}

\begin{enumerate}
\item Let $Y\in Z(C)$ and $T\in M_{n}(\mathbb{F})$. Then
\begin{equation*}\label{eqn1}
        \left.\begin{array}{c}
           
           CT=TC+YE\\
	\end{array}\right. \Leftrightarrow \ Y=0,\ T\in Z(C).
\end{equation*}

\item Let $X\in Z(C)$ and $T\in M_{n}(\mathbb{F})$. Then
\begin{equation*}\label{eqn12}
        \left.\begin{array}{c}
           TC=CT+EX\\
         \end{array}\right. \Leftrightarrow\ X=0,\ T\in Z(C).
\end{equation*}
\end{enumerate}
\end{corollary}

The next theorem is the main result of this subsection. We give a characterization of   the centralizer of a generalized Jordan block.

\begin{theorem}[Centralizer of a generalized Jordan block]\label{corocentraGJ}

Let $G\in M_{s\ell}(\mathbb{F})$ be a generalized Jordan block, $m_{G}=p^{\ell}$, $p$ irreducible and $\deg(p)=s$. Then, 
the centralizer $Z(G)$ of $G$ is

\begin{equation*}\label{centralizer}
\left\{\left[\begin{array}{ccccc}
X_{1,1} & 0 & \ldots &  0  & 0 \\
X_{2,1}&X_{1,1}&\ldots& 0 & 0 \\
\vdots&\ddots & \ddots& \vdots & \vdots \\
X_{\ell-1,1}&X_{\ell-2,1}&\ddots&X_{1,1}&0\\
X_{\ell,1}&X_{\ell-1,1}&\dots&X_{2,1}&X_{1,1}
\end{array}\right]
, \begin{array}{lc}  X_{1,1}\in Z(C),\\
    \vspace{0.3mm}
    X_{i,1}=X_{i,1}'+\tilde{X}_{i-1,1},\\ 
    \vspace{0.3mm}
    X_{i,1}'\in Z(C),\\ 
    \vspace{0.3mm}
    i=2,\ldots,\ell.\end{array}
\right\},
\end{equation*}
where $\tilde{X}_{i-1,1}$ is defined as in (\ref{deftildeX}).
\end{theorem}

\begin{proof}
 For $\ell=1$ the result is immediate. Assume that $\ell\geq 2$. We prove the theorem by induction on $k=2,\ldots,\ell$. For $k=2$ 
 it is straightforward to see that
\[Z\left(\left[\begin{array}{cc}
C & 0 \\ E & C
 \end{array}\right]\right)=
\left\{
 \left[\begin{array}{cc}
X & 0 \\ T & X
 \end{array}\right]:  \ T=T'+\tilde{X}, \quad X, T'\in Z(C)
\right\}.
\]

Assume that the property is true for $k$ and let us prove that it is satisfied for $k+1$.

Let us write
$$G_{k+1}=\left[\begin{array}{cc}G_{k}&0\\
E_{k}&C          \end{array}\right],$$
where $E_{k}=\left[\begin{array}{cccc}0&\ldots&0& E \end{array}\right]$.
Assume that $X_{k+1}\in Z(G_{k+1})$ and write
$$X_{k+1}= [X_{i,j}]_{1\leq i,j\leq k+1}=
\left[\begin{array}{c|c}
X_{k}&\begin{array}{c}X_{1,k+1}  \\ \ldots \\ X_{k,k+1}  \end{array}\\
\hline
\begin{array}{ccc}X_{k+1,1}  & \ldots &X_{k+1,k} \end{array}&X_{k+1,k+1} \end{array}\right], \quad   X_{i,j}\in M_{s}(\mathbb{F}).$$
Then,
\[ X_{k+1}G_{k+1}=G_{k+1}X_{k+1}\]
if and only if the following equations (\ref{eq1})-(\ref{eq4}) are satisfied
\begin{equation} \label{eq1}
 G_{k}X_{k}=X_{k}G_{k}+
 \left[\begin{array}{cccc}0&\ldots&0&X_{1,k+1}E\\
\vdots& &\vdots&\vdots\\
0&\ldots&0& X_{k,k+1}E \end{array}\right],
\end{equation}
\begin{equation} \label{eq2}
 G_{k}\left[\begin{array}{c}X_{1,k+1}\\ \vdots \\ X_{k,k+1} \end{array}\right]=
 \left[\begin{array}{c}X_{1,k+1}C\\ \vdots \\ X_{k,k+1}C \end{array}\right],
 \end{equation}
\begin{equation} \label{eq3}
 \left\{\begin{array}{ccc}
EX_{k,1}+CX_{k+1,1} & = &X_{k+1,1}C+X_{k+1,2}E, \\
\vdots  &\vdots  &  \vdots\\
 EX_{k,k}+CX_{k+1,k} & = &X_{k+1,k}C+X_{k+1,k+1}E,
 \end{array}\right.
 \end{equation}
 \begin{equation} \label{eq4}
 EX_{k,k+1}+CX_{k+1,k+1}=X_{k+1,k+1}C.
 \end{equation}
From equations (\ref{eq2}) and (\ref{eq4})
we have
\begin{equation*}
 \left\{\begin{array}{rcc}
 CX_{1,k+1} & = & X_{1,k+1}C\\
 EX_{1,k+1}+CX_{2,k+1}& = &X_{2,k+1}C\\
& \vdots &\\
 EX_{k-1,k+1}+CX_{k,k+1}& = &X_{k,k+1}C \\
 EX_{k,k+1}+CX_{k+1,k+1}& = &X_{k+1,k+1}C
 \end{array}\right.
\end{equation*}
and as a consequence of Corollary~\ref{coreq} we obtain $X_{1,k+1}=\ldots=X_{k,k+1}=0$ and $X_{k+1,k+1}\in Z(C)$.

Now, equation (\ref{eq1}) reduces to  $G_{k}X_{k}=X_{k}G_{k}$. Applying the induction hypothesis we can write
\begin{equation*}
 X_{k}=\left[\begin{array}{ccccc}
X_{1,1} & 0&\dots&\dots&0\\
X_{2,1}&X_{1,1}&0&\dots&0\\
X_{3,1}&X_{2,1}&X_{1,1}&\dots&0\\
\vdots &\vdots &\vdots &\ddots&\vdots \\
X_{k,1} &X_{k-1,k}&X_{1,k-2}&\dots &X_{1,1}
\end{array}\right],
 \end{equation*}
where $X_{1,1}\in Z(C)$ and $X_{i,1}=X_{i,1}'+\tilde{X}_{i-1,1}, X_{i,1}'\in Z(C)$
and $\tilde{X}_{i-1,1}$ is as in (\ref{deftildeX}), for $i=2,\ldots,k$.
It allows us to  rewrite equations (\ref{eq3}) as follows
\begin{equation*}
 \left\{\begin{array}{c}  CX_{k+1,k+1}=X_{k+1,k+1}C\\
                          EX_{1,1}+CX_{k+1,k}=X_{k+1,k}C+X_{k+1,k+1}E\\
                           EX_{2,1}+CX_{k+1,k-1}=X_{k+1,k-1}C+X_{k+1,k}E\\
                           \vdots\\
                           EX_{k-1,1}+CX_{k+1,2}=X_{k+1,2}C+X_{k+1,3}E\\
                         EX_{k,1}+CX_{k+1,1}=X_{k+1,1}C+X_{k+1,2}E
                          \end{array}\right.\end{equation*}
As $X_{1,1}$ and $X_{k+1,k+1}\in Z(C)$, by Lemma~\ref{lemaequationsxtilde} the second of these equations implies
that $X_{k+1,k+1}=X_{1,1}$ and $X_{k+1,k}=X'_{k+1,k}+\tilde{X}_{1,1}$ with $X_{k+1,k}'\in Z(C)$.
Now, from the third equation and Lemma~\ref{lemaequations} we obtain that $X_{k+1,k}=X_{2,1}$
and $X_{k+1,k-1}=X'_{k+1,k-1}+\tilde{X}_{2,1}$. Proceeding in the same way we obtain the desired result.
\end{proof}

The following corollary generalizes Corollary~\ref{coro:dim}.
\begin{corollary}\label{coro:dimGk}
Under the hypotesis of Theorem~\ref{corocentraGJ} we have that
$$\dim (Z(G))=\ell s.$$
\end{corollary}

\subsection{The centralizer of a generalized Jordan form}\label{The centralizer}

In order to obtain the centralizer of a generalized Jordan form we need to prove first two technical results.

\begin{lemma}\label{thecentG1G2col}
Let  $G_1\in M_{sa}(\FF), G_2\in M_{sb}(\FF)$, $a\geq b$, be two generalized Jordan blocks.
Let $T=\left[T_{i,j}\right]\in M_{sa\times sb}(\mathbb{F})$ be a block matrix with $T_{i,j}\in M_{s}(\FF)$
such that
\begin{equation} \label{conmG1G2}
G_1T=TG_2.
\end{equation}
Then,
 \begin{equation*} 
 T=\begin{bmatrix}
 0 \\
 T_1
\end{bmatrix},
\end{equation*}
with $T_1\in Z(G_{2})$.
\end{lemma}

\begin{proof}
Denoting $T=[Y_{i,j}]_{i=1,\ldots, a, j=1, \ldots, b}$, equation~(\ref{conmG1G2})  is
$$
\begin{bmatrix}
CY_{1,1} & \ldots & CY_{1,b-2} & CY_{1,b-1} & CY_{1,b} \\
EY_{1,1} +CY_{2,1} & \ldots & EY_{1,b-2} +CY_{2,b-2} & EY_{1,b-1} +CY_{2,b-1} & EY_{1,b} +CY_{2,b}  \\
\vdots & &\vdots & \vdots & \vdots \\
EY_{a-3,1}+ CY_{a-2,1} & \ldots &  EY_{a-3,b-2}+ CY_{a-2,b-2}  & EY_{a-3,b-1} +CY_{a-2,b-1} & EY_{a-3,b} +CY_{a-2,b}\\
EY_{a-2,1}+ CY_{a-1,1} & \ldots &  EY_{a-2,b-2}+ CY_{a-1,b-2} & EY_{a-2,b-1} +CY_{a-1,b-1} & EY_{a-2,b} +CY_{a-1,b}\\
EY_{a-1,1}+ CY_{a,1} & \ldots &  EY_{a-1,b-2}+ CY_{a,b-2} & EY_{a-1,b-1} +CY_{a,b-1} & EY_{a-1,b} +CY_{a,b}
\end{bmatrix}=$$
$$
=\begin{bmatrix}
Y_{1,1}C+Y_{1,2}E & \ldots &  Y_{1,b-2}C+Y_{1,b-1}E  & Y_{1,b-1}C+Y_{1,b}E &Y_{1,b}C \\
Y_{2,1}C+Y_{2,2}E & \ldots & Y_{2,b-2}C+Y_{2,b-1}E & Y_{2,b-1}C+Y_{2,b}E &Y_{2,b}C \\
\vdots & & \vdots &  \vdots & \vdots\\
Y_{a-2,1}C+Y_{a-2,2}E & \ldots & Y_{a-2,b-2}C+Y_{a-2,b-1}E & Y_{a-2,b-1}C+Y_{a-2,b}E &Y_{a-2,b}C\\
Y_{a-1,1}C+Y_{a-1,2}E & \ldots & Y_{a-1,b-2}C+Y_{a-1,b-1}E & Y_{a-1,b-1}C+Y_{a-1,b}E &Y_{a-1,b}C\\
Y_{a,1}C+Y_{a,2}E & \ldots & Y_{a,b-2}C+Y_{a,b-1}E & Y_{a,b-1}C+Y_{a,b}E &Y_{a,b}C
\end{bmatrix}.$$
Identifying the  block components of the last column  and taking into account Corollary \ref{coreq}, we  successively obtain that
$$Y_{1,b}=0, \quad Y_{2,b}=0, \quad \ldots \quad Y_{a-1,b}=0, \quad Y_{a,b}\in Z(C).$$
Then, identifying the  block components of the last but one column  and taking into account Corollary \ref{coreq} and Lemma \ref{lemaequations},
we successively obtain that
$$Y_{1,b-1}=0, \quad Y_{2,b-1}=0, \quad \ldots \quad Y_{a-2,b-1}=0, \quad Y_{a-1,b-1}=Y_{a,b}, $$
and
$$Y_{a,b-1}=Y_{a,b-1}'+\tilde{Y}_{a,b}, \quad  Y_{a,b}'\in Z(C).$$
Identifying the  block components of the third block column starting from the end and taking into account Corollary \ref{coreq}, Lemma \ref{lemaequations} and Lemma \ref{lemaequationsxtilde}, we obtain
$$Y_{1,b-2}=0, \quad Y_{2,b-2}=0, \quad \ldots \quad Y_{a-3,b-2}=0, \quad Y_{a-2,b-2}=Y_{a,b}, $$
and
$$Y_{a-1,b-2}=Y_{a,b-1}, \quad  Y_{a,b-2}=Y_{a,b-2}'+\tilde{Y}_{a,b}, \quad Y_{a-1,b-2}''\in Z(C).$$

Proceeding analogously down the remaining columns we obtain
$$
T=
\begin{matriz}{ccccc}
0  & 0 &  \ldots & 0 & 0 \\
\vdots & \vdots &  & \vdots  & \vdots \\
0 &  0 & \ldots & 0  & 0\\
Y_{a,b}  & 0& \ldots & 0& 0  \\
Y_{a,b-1}  & Y_{a,b} & \ldots & 0 & 0\\
\vdots & \vdots & \ddots & \vdots & \vdots \\
Y_{a,2}  & Y_{a,3}& \ldots & Y_{a,b}   & 0\\
Y_{a,1}  & Y_{a,2}&  \ldots & Y_{a,b-1} & Y_{a,b}
\end{matriz},$$
with $Y_{a,j}=Y_{a,j}'+\tilde{Y}_{a,j-1}$ and $Y_{a,b}, Y_{a,j}' \in Z(C)$ for $j=1, \ldots, b-1$ as desired.
\end{proof}

The next lemma can be proved in a similar way.

\begin{lemma}\label{thecentG1G2row}
Let $G_1, G_2$ be two generalized Jordan blocks as in (\ref{gJordanblock}), $G_1\in M_{sa}(\FF), G_2\in M_{sb}(\FF)$, $a\leq b$.
Let $T=[T_{i,j}]_{i=1, \ldots, a, j=1, \ldots, b}$ be a block matrix with $T_{i,j}\in M_{s}(\FF)$ such that
\begin{equation*} \label{conmG1G2_2}
G_1T=TG_2.
\end{equation*}
Then,
 \begin{equation*} 
 T=\begin{bmatrix}
 T_1 & 0
\end{bmatrix},
\end{equation*}
with $T_1\in Z(G_{1})$.
\end{lemma}


\begin{theorem}[Centralizer of a generalized Jordan form]\label{the:centGJ}
Let $G=\diag(G_1, \ldots, G_m)$ be a generalized Jordan matrix. 
Let $p_{G}=p^{r}$ be its characteristic polynomial, $p$ irreducible, $m_G=p^{r_1}, \ r_1\leq r$ and $\deg{(p)}=s$.
Let $\alpha=(\alpha_1, \ldots, \alpha_m)$ be the generalized Segre characteristic of $G$. If $X\in Z(G)$, then
$$X=[X_{i,j}]_{i,j=1, \ldots, m},$$
where $X_{i,j}\in M_{s\alpha_{i}\times s\alpha_{j}}(\mathbb{F})$  are block lower triangular Toeplitz matrices of the form:
\begin{enumerate}
 \item[1)] If $\alpha_i=\alpha_j$, then $X_{i,i}\in Z(G_{i})$.
\item[2)] If $\alpha_i<\alpha_j$, then 
\[X_{i,j}=
\left[\begin{array}{cc}
X_{i,i} & 0
\end{array}\right],\]
where $X_{i,i}\in Z(G_{i})$.
\item[3)] If $\alpha_i>\alpha_j$, then 
\[X_{i,j}=
\left[\begin{array}{c}
0 \\
X_{j,j}
\end{array}\right],\]
where $X_{j,j}\in Z(G_{j})$. 
\end{enumerate}

\end{theorem}

\begin{proof}
Let $X\in Z(G)$ and assume that  $X=[X_{i,j}]_{i,j=1, \ldots, m}$. Then, the block components of $X$ satisfy the following equations:
$$G_iX_{i,j}=X_{i,j}G_j, \quad i,j=1, \ldots, m.$$
The structure of the blocks $X_{i,j}$ for $i,j=1, \ldots, m$ is a direct consequence of Theorem \ref{corocentraGJ} and Lemmas \ref{thecentG1G2col} and \ref{thecentG1G2row}.
\end{proof}

\begin{example}\label{exem:centraJordan}
Let $G$ be a generalized Jordan matrix with minimal polynomial $m_{G}= p^{5}$, where $p$ is an irreducible polynomial of $\deg(p)>1$ 
and whose companion matrix is $C$ as in~(\ref{companion}). Let  $\alpha=(5,4,3,1,1)$ be its generalized Segre characteristic, i.e. 
the generalized Jordan form of $G$ is
$$G=
\begin{matriz}{ccccc|cccc|ccc|c|cc}
 C &  & & & & & & & & & & & &  \\
 E & C & & & & & & & & & & & &  \\
 & E & C & & & & & & & & & & & \\
 & & E & C & & & & & & & & & & \\
 & & & E & C & & & & & & & && \\
 \hline
 & & & & & C & & & & & & & &  \\
 & & & & & E & C & & & & & & &  \\
 & & & & & & E & C & & & & & &  \\
 & & & & & & & E & C & & & & &\\
 \hline
 & & & & & & & & & C & & & & \\
 & & & & & & & & & E & C & & & \\
 & & & & & & & & & & E & C & & \\
 \hline
 & & & & & & & & & & & & C & \\
 \hline
  & & & & & & & & & & & & & C
\end{matriz} $$

In this case a matrix $X\in Z(G)$ has the following form

$$X=\begin{matriz}{ccccc|cccc|ccc|c|c}
A_{1} &  &  &  &  &  &  &  &  &  &  &  &  & \\
A_{2} &A_{1} &  & &  &I_{1} &  & &  &  &  &  &  &  \\
A_{3} &A_{2} &A_{1} &  &   &I_{2} &I_{1} &  &  &L_{1} &  &  &  & \\
A_{4} & A_{3} &A_{2} & A_{1} &  & I_{3} &I_{2} & I_{1} &  & L_{2}&L_{1}  &  &  &  \\
A_{5} & A_{4} & A_{3}& A_{2} &A_{1} & I_{4}& I_{3} & I_{2} &I_{1} & L_{3}& L_{2}  &L_{1}&Q_{1} &W_{1}  \\
\hline
H_{1} &  &  &  &  & B_{1}&   &  &  &   &  &  &  & \\
H_{2} & H_{1}  &  &  &  & B_{2}& B_{1} &  &  & M_{1} &  &  &  & \\
H_{3} & H_{2} & H_{1}  &  &  &  B_{3}& B_{2}& B_{1}& & M_{2}&M_{1} & & &  \\
H_{4} & H_{3} &  H_{2}&H_{1} &  & B_{4}& B_{3}& B_{2} &B_{1} & M_{3}& M_{2}& M_{1}& R_{1}  & X_{1}  \\
\hline
J_{1} &  &  &  &  &K_{1} &  &  &  &C_{1} &  &  &  &  \\
J_{2} &J_{1}  &  &  &  & K_{2}& K_{1}  &  &  &  C_{2}& C_{1}  &  &  &  \\
J_{3} &J_{2}  &J_{1}  &  &  &  K_{3}& K_{2} & K_{1} &  & C_{3} & C_{2}&C_{1} & S_{1}&Y_{1} \\
\hline
N_{1} &  &  &  &  &  O_{1}  &  &  &  &P_{1}  &  &  &  D_{1}&G_{1}  \\
\hline
T_{1} &  &  &  &  & U_{1} &  &  &  & V_{1} &  &  &  F{1}& E_{1}
\end{matriz},$$
where $A_{1}\in Z(C)$ and for $i=2,\ldots,5$, $A_{i}=A'_{i}+\tilde{A}_{i-1}$ with $A_{i}'\in Z(C)$ and $\tilde{A}_{i-1}$ is defined 
as in~(\ref{deftildeX}).
An analogous pattern occurs in each block.
\end{example}


\section{The centralizer of a matrix over a perfect field}\label{sec:Perfect}

When $\FF$ is a perfect field we can replace the matrices $E$ by identity matrices $I$ in the generalized Jordan form, obtaining 
the generalized Jordan canonical form of the first kind (see Theorem \ref{jordan_separable}). In fact, a perfect 
field is not a necessary condition, it is enough that the polynomial $p$ is separable.

\medskip

In what follows, we calculate the centralizer of a generalized Jordan matrix of the first kind. Although it is not a particular case of the centralizer of a generalized Jordan form, we can easily derive the technical results we need to 
obtain the centralizer in this case from those obtained in Section \ref{gJm}. 

An example of this case is the centralizer of the real Jordan form.

\medskip


The following results are variants of Lemma~\ref{lemaequations} and Corollary~\ref{coreq}, respectively,  adapted to this case.
	
\begin{lemma}\label{lemaequations_sep}
 Let $C\in M_{n}(\mathbb{F})$ be a companion matrix as in (\ref{companion}).
 Let $X,Y\in Z(C)$ and $T \in M_{n}(\mathbb{F})$. Then
 \begin{equation*}\label{eqn2_sep}
        \left.\begin{array}{c}
           X+CT=TC+Y\\
           \end{array}\right.\Leftrightarrow\ X=Y,  \  T \in Z(C). \end{equation*}

\end{lemma}

This result is, in fact, an immediate consequence of the next lemma.

\begin{lemma}\label{coreq_sep}
 Let $C\in M_{n}(\mathbb{F})$ be a companion matrix as in (\ref{companion}). 
Let $Y\in Z(C)$ and $T \in M_{n}(\mathbb{F})$. Then
\begin{equation*}\label{eqn1_sep}
        \left.\begin{array}{c}
           CT=TC+Y\\
	\end{array}\right.\Leftrightarrow \ Y=0,\ T\in Z(C).
\end{equation*}

\end{lemma}

\begin{proof}
By Lemma~\ref{lemmacentracom} we have that
\[Y \in Z(C) \Rightarrow  Y= \left[\begin{array}{cccc} Y_{\ast 1} & CY_{\ast 1} & \ldots & C^{n-1}Y_{\ast 1}  \end{array}\right].\]
$$CT=\left[\begin{array}{ccc} CT_{\ast 1} & \ldots & CT_{\ast n} \end{array}\right]$$
$$TC+Y= \left[\begin{array}{cccc} T_{\ast 2} & \ldots & T_{\ast n} & -c_0T_{\ast 1}- \ldots -c_{n-1}T_{\ast n} \end{array}\right] +$$
$$\left[\begin{array}{cccc} Y_{\ast 1} & CY_{\ast 1} & \ldots & C^{n-1}Y_{\ast 1}\end{array} \right].$$
Identifying the two expressions we obtain:
\begin{equation*} \label{X_sep}
 \left.\begin{array}{l}
T_{\ast 2}=CT_{\ast 1} + Y_{\ast 1} \\
T_{\ast 3}=CT_{\ast 2} + CY_{\ast 1}=+C^2T_{\ast 1} +2CY_{\ast 1}\\
\ldots \\
T_{\ast i}=C^{i-1}T_{\ast 1} + (i-1)C^{i-2}Y_{\ast 1}\\
\ldots \\
T_{\ast n}= C^{n-1}T_{\ast 1} + (n-1)C^{n-2}Y_{\ast 1}\\
-c_0T_{\ast 1}- \ldots -c_{n-1}T_{\ast n}=CT_{\ast n} + C^{n-1}Y_{\ast 1}
\end{array}\right\}
\end{equation*}
Replacing the values of all $T_{\ast i}$ in the last equation we have:

\noindent
$-c_0T_{\ast 1}- c_1(CT_{\ast 1}+ Y_{\ast 1}) \ldots  -c_{i-1}(C^{i-1}T_{\ast 1} + (i-1)C^{i-2}Y_{\ast 1})   \ldots -c_{n-1}(C^{n-1}T_{\ast 1} + (n-1)C^{n-2}Y_{\ast 1})=
C( C^{n-1}T_{\ast 1} + (n-1)C^{n-2}Y_{\ast 1}+ C^{n-1}Y_{\ast 1},$

\noindent
therefore,
$$p(C)T_{\ast 1}=p'(C)Y_{\ast 1}.$$
Observe that $p(C)=0$ and since $p$ is separable $p'(C)\neq 0$. Moreover, $p'(C) \in Z(C)$. By Corollary \ref{corodim} we have that $\det(p'(C))\neq 0$. Hence,  
$Y_{\ast 1}=0$. 

The converse is trivial.
\end{proof}

\begin{theorem}[Centralizer of a generalized Jordan block of the first kind]\label{corocentraGJ_sep}

Let  $G\in M_{s\ell}(\mathbb{F})$ be a generalized Jordan block of the first kind with $m_{G}=p^{\ell}$, $\deg(p)=s$, $p$ irreducible 
and separable. Then, the centralizer $Z(G)$ of $G$ is
\begin{equation*}\label{centralizer_sep}
\left\{
\left[\begin{array}{cccc}
X_{1} & 0 & \ldots &  0 \\
X_{2}&X_{1}&\ldots& 0 \\
\vdots&\vdots & \ddots& \vdots\\

X_{\ell}&X_{\ell-1}&\dots&X_{1}
 \end{array}\right]
, \  X_{i}\in Z(C),\ i=1,\ldots,\ell
\right\}.
\end{equation*}

\end{theorem}

\begin{proof}
 For $\ell=1$ the result is immediate. Assume that $\ell\geq 2$. We prove the theorem by induction on $k=2,\ldots,\ell$.  For $k=2$, it can be proved as a consequence of Lemmas ~\ref{lemaequations_sep},~\ref{coreq_sep}. 
Assume that the hypothesis is true for $k$. To prove that it is also true for $k+1$ it is enough to follow step by step the proof  
of Theorem \ref{corocentraGJ}, replacing $E$ by the identity matrix.
\end{proof}

 Observe that the centralizer in this case is analogous to the centralizer obtained in Theorem \ref{corocentraGJ}, but now there is no dependency between 
 the blocks of a lower diagonal and those of the diagonal immediately above it.

In order to obtain the centralizer of a generalized Jordan form of the first kind, we need to translate Lemmas~\ref{thecentG1G2col} and~\ref{thecentG1G2row}
to the case where the polynomial $p$ is separable. Their proofs are analogous to those of 
Lemmas~\ref{thecentG1G2col} and~\ref{thecentG1G2row}.
 
\begin{lemma}\label{thecentG1G2col_sep}
Let  $G_1\in M_{sa}(\FF), G_2\in M_{sb}(\FF)$, $a\geq b$ be two generalized Jordan blocks of the first kind.
Let $T=\left[T_{i,j}\right]\in M_{sa\times sb}(\mathbb{F})$ be a block matrix with $T_{i,j}\in M_{s}(\mathbb{F})$ such that
\begin{equation*} \label{conmG1G2_sep}
G_1T=TG_2.
\end{equation*}
Then,
 \begin{equation*} 
 T=\begin{bmatrix}
 0 \\
 T_1
\end{bmatrix},
\end{equation*}
with $T_1\in Z(G_{2})$.
\end{lemma}

\begin{lemma}\label{thecentG1G2row_sep}
Let $G_1, G_2$ be two generalized Jordan blocks of the first kind, $G_1\in M_{sa}(\FF), G_2\in M_{sb}(\FF)$, $a\leq b$.
Let $T=\left[T_{i,j}\right]\in M_{sa\times sb}(\mathbb{F})$ be a block matrix with $T_{i,j}\in M_{s}(\FF)$ such that
\begin{equation*} \label{conmG1G2_2_sep}
G_1T=TG_2.
\end{equation*}
Then,
 \begin{equation*} 
 T=\begin{bmatrix}
 T_1 & 0
\end{bmatrix},
\end{equation*}
with $T_1\in Z(G_{1})$.
\end{lemma}

\begin{theorem}\label{the:centGJ_sep}
Let $G=\diag(G_1, \ldots, G_m)$ be a generalized Jordan matrix of the first kind. 
Let $p_{G}=p^{r}$ be its characteristic polynomial, $p$ be irreducible and separable,
$m_G=p^{r_1}, \ r_1\leq r$ and $\deg{p}=s$. Let $\alpha=(\alpha_1, \ldots, \alpha_m)$ be the generalized Segre characteristic 
of $G$. If $X\in Z(G)$, then
$$X=[X_{i,j}]_{i,j=1, \ldots, m},$$
where $X_{i,j}\in M_{s\alpha_{i}\times s\alpha_{j}}(\mathbb{F})$  are block lower triangular Toeplitz matrices of the form:
\begin{enumerate}
 \item[1)] If $\alpha_{i}=\alpha_{j}$, then $X_{i,i}\in Z(G_{i})$.
\item[2)] If $\alpha_{i}<\alpha_{j}$, then
\[X_{i,j}=
\left[\begin{array}{cc}
X_{i,i} & 0
\end{array}\right],\]
with $X_{i,i}\in Z(G_{i})$.
\item[3)] If $\alpha_{i}>\alpha_{j}$, then
\[X_{i,j}=
\left[\begin{array}{c}
0 \\
X_{j,j}
\end{array}\right],\]
with $X_{j,j}\in Z(G_{j})$.
\end{enumerate}
\end{theorem}

\begin{proof}
The result follows straightforward from Theorem~\ref{corocentraGJ_sep} and 
Lemmas~\ref{thecentG1G2col_sep} and \ref{thecentG1G2row_sep}.
\end{proof}


\section{Centralizer of a generalized Weyr form}\label{sec:centraWeyr}

In this section we compute the centralizer of a generalized Weyr matrix. Let $W$ be a matrix of the form

\begin{equation*}
W=\left[\begin{array}{cccccc}
\vspace{-0.15cm}
W_1&E_2& \ldots & 0 &0\\
0&W_2&\ddots& 0 &0\\
\vdots & \vdots & \ddots & \ddots & \vdots \\
0 & 0 & \ldots &W_{\alpha_{1}-1}&E_{\alpha_{1}}\\
0 & 0 & \ldots & 0 &W_{\alpha_{1}}
\end{array}\right],\end{equation*}
with $W_{i}$ and $E_{i}$ as in~(\ref{weyrform}).

\begin{theorem}\label{weyrformcentra}
Let $W$ be the generalized Weyr matrix with generalized Segre characteristic $\alpha=(\alpha_1, \ldots, \alpha_m)$. Let
$\tau=(\tau_1, \ldots, \tau_{\alpha_{1}})$ be the conjugate partition of $\alpha$. Then, if $K\in Z(W)$,

\begin{equation}\label{centralitzadorK} K=\left[\begin{array}{cccccc}
K_{1,1} & K_{1,2}  & \ldots& K_{1,\alpha_{1}-1}&K_{1,\alpha_{1}}  \\
0 & K_{2,2} &  \ldots & K_{2,\alpha_{1}-1} & K_{2,\alpha_{1}} \\		
 \vdots &\vdots&\ddots& \vdots   & \vdots  \\
0 & 0  & \ldots & K_{\alpha_{1}-1,\alpha_{1}-1} & K_{\alpha-{1}-1,\alpha_{1}}\\
0 & 0&\ldots & 0 & K_{\alpha_{1},\alpha_{1}}
\end{array}\right],\end{equation}
built according to the following recursive construction

\begin{enumerate}
\item
$K_{\alpha_{1},\alpha_{1}}$ is a block matrix of $\tau_{\alpha_{1}}\times  \tau_{\alpha_{1}}$ independent blocks of $Z(C)$.

\item The  blocks on the main diagonal for $i=1, \ldots, \alpha_{1}-1$ are of the form
  $$K_{i,i}= \left[\begin{array}{cc}
K_{i+1,i+1} & Y_{i,i}\\
0 & X_{i,i}
\end{array}\right],$$
where $X_{i,i}$ is composed by  
 $(\tau_{i}-\tau_{i+1})\times (\tau_{i}-\tau_{i+1})$ independent blocks of $Z(C)$ and $Y_{i,i}$ is composed  
by $\tau_{i+1}\times (\tau_{i}-\tau_{i+1})$ independent blocks of $Z(C)$.

\medskip

\item The  blocks on the last column for $i=1, \ldots, \alpha_{1}-1$ are of the form
 $$K_{i,\alpha_{1}}= \left[\begin{array}{c}
Y_{i,\alpha_{1}}\\
 X_{i,\alpha_{1}}
\end{array}\right],$$
where $X_{i,\alpha_{1}}$ is a block matrix of $(\tau_{i}-\tau_{i+1})\times \tau_{\alpha_{1}}$ independent blocks of $Z(C)$ and 
$Y_{i,\alpha_{1}}=Y_{i,\alpha_{1}}'  +\left(\begin {array}{c} \tilde{Y}_{i+1,\alpha_{1}} \\
\tilde{X}_{i+1,\alpha_{1}}\end{array} \right) $  
where $Y_{i,\alpha_{1}}' $ is composed by $\tau_{i+1}\times \tau_{\alpha_{1}}$ independent blocks of $Z(C)$ and 
$\tilde{Y}_{i+1,\alpha_{1}}$, $\tilde{X}_{i+1,\alpha_{1}}$ are composed by blocks defined as in (\ref{deftildeX}).
Notice that $Y_{\alpha_{1},\alpha_{1}}=K_{\alpha_{1},\alpha_{1}}$.

\item For $i, j=1, \ldots, s-1,\quad i\leq j, $
$$K_{i,j}= \left[\begin{array}{cc}
K_{i+1,j+1} & Y_{i,j}\\
0 & X_{i,j}
\end{array}\right],$$
where $X_{i,j}$ is a block matrix  of $(\tau_{i}-\tau_{i+1})\times  (\tau_{j}-\tau_{j+1})$ independent blocks of $Z(C)$,
and $Y_{i,j}=Y_{i, j}' + \left(\begin {array}{c}\tilde{Y}_{i+1,j}\\
 \tilde{X}_{i+1,j}\end{array} \right) $ where $Y_{i,j}'$  is composed by $\tau_{i+1}\times  (\tau_{j}-\tau_{j+1})$ independent blocks 
 of $Z(C)$ and $\tilde{Y}_{i+1,j}, \tilde{X}_{i+1,j} $ are composed by blocks defined as in (\ref{deftildeX}).

\end{enumerate}
\end{theorem}
\begin{proof}
Let $G$ be a matrix similar to $W$ in generalized Jordan form. If $P$ is the matrix described in~(\ref{permutationmatrix}) and 
$X\in Z(G)$, then $P^{-1}XP=K$.
\end{proof}

\begin{remark}
 If $p$ is separable, the block structure of $K\in Z(W)$ is the same as (\ref{centralitzadorK}) but every block component is in $Z(C)$. 
\end{remark}

\begin{example}\label{exem:centraWeyr}

Following with Example~\ref{exem:centraJordan}, the  Weyr characteristic of $G$ is
$\tau=(5,3,3,2,1).$ This partition $\tau$ gives us the number of blocks of the diagonal blocks in the generalized Weyr form.

The generalized Weyr form is
\begin{equation*}
W=
\begin{matriz}{ccccc|ccc|ccc|cc|c}
 C & & & & & E & & & & & & & &  \\
 & C & & & & & E & & & & & & &  \\
 & & C & & & & & E& & & & & & \\
 & & &C & & & & & & & & & & \\
 & & & &C & & & & & & & && \\
\hline
 & & & & & C & & & E & & & & &  \\
 & & & & & & C & & & E & & & &  \\
 & & & & & & & C & & & E & & &  \\
\hline
 & & & & & & & & C & & & E & &\\
 & & & & & & & & & C & & & E & \\
 & & & & & & & & & & C & & & \\
\hline
 & & & & & & & & & & & C & & E\\
 & & & & & & & & & & & &C & \\
 \hline
 & & & & & & & & & & & & & C
\end{matriz},
\end{equation*}
and a matrix $K\in Z(W)$ has the form

$$K= \begin{matriz}{ccccc|ccc|ccc|cc|c}
A_{1} & I_{1} &L_{1} &Q_{1} &W_{1} &A_{2} &I_{2} &L_{2} &A_{3} &I_{3} &L_{3} &A_{4} &I_{4} & A_{5} \\
& B_{1} & M_{1}& R_{1}& X_{1}& H_{1}& B_{2}& M_{2}& H_{2}& B_{3}& M_{3}& H_{3}& B_{4}& H_{4} \\
&  & C_{1}& S_{1}& Y_{1}& & K_{1}& C_{2}& J_{1}& K_{2}& C_{3}& J_{2}&K_{3} & J_{3} \\
 &  & & D_{1} &G_{1} & & & & & & P_{1}& & O_{1}& N_{1} \\
 &  & & F_{1} & E_{1} & & & & & & V_{1}& &U_{1} &T_{1}  \\
 \hline
 &  & & & & A_{1}& I_{1}& L_{1}& A_{2}& I_{2}& L_{2}& A_{3}& I_{3}& A_{4} \\
 &  & & & & & B_{1}&M_{1} & H_{1}& B_{2}& M_{2}& H_{2}& B_{3}& H_{3}\\
 &  & & & & & & C_{1}& &K_{1} & C_{2}& J_{1}& K_{2}& J_{2} \\
 \hline
 &  & & & & & & & A_{1}& I_{1}&L_{1} & A_{2}&I_{2} & A_{3} \\
&  & & & & & & & & B_{1}& M_{1}& H_{1}& B_{2}& H_{2} \\
&  & & & & & & & & & C_{1}& & K_{1}& J_{1} \\
\hline
&  & & & & & & & & & &A_{1} & I_{1}& A_{2}\\
&  & & & & & & & & & & &B_{1} & H_{1}\\
\hline
 &  & & & & & & & & & & & & A_{1}
    \end{matriz},$$
where the blocks in this matrix satify the same relations as in Example~\ref{exem:centraJordan}.
\end{example}

\subsection{Determinant of the centralizer}\label{sec:Deter}


We proved in Theorem \ref{weyrformcentra} that the elements of the centralizer of the generalized Weyr canonical form are block upper  triangular matrices.
As a consequence, the  determinant of $K\in Z(W)$ can be computed as the product of the determinants of the diagonal blocks (see~\cite{Det13} for the centralizer of a Weyr form)
 \begin{equation*}
\det(K)= \det(K_{1,1})\det(K_{2,2})\ldots \det(K_{r,r}).
\end{equation*}
Hence, the diagonal blocks are key to characterize the automorphisms of the centralizer. An important 
application of this property will be the characterization of hyperinvariant and characteristic lattices of 
the endomorphism (see \cite{MMP13} for the case when $p$ splits over $\mathbb{F}$).   

In Example \ref{exem:centraJordan} (also Example \ref{exem:centraWeyr}), the expression of the determinant for those elements 
of the centralizer results in
$$
\det(K)= \det(X)=\det(A_{1})^5 \det(B_{1})^4  \det(C_{1})^3
 \det\left[\begin{matrix} D_{1}&G_{1}\\
                                                     F_{1}&E_{1} \end{matrix} \right].    
$$

\begin{remark}
Notice that the formula is exactly the same in the separable and nonseparable cases, as the block components of the  diagonal  blocks of a matrix $K\in Z(W)$ are elements in $Z(C)$. Therefore, 
the condition for a matrix in the centralizer to be an automorphism is exactly the same in both cases.
\end{remark}

\subsection{Dimension of the centralizer}
As a consequence of Corollary~\ref{coro:dimGk}, to compute the dimension of the centralizer we must take into account that each block has dimension equal to $s$. Then, according to the Segre and Weyr characteristics we have that 
$$\dim (Z(G))= s(\alpha_{1}+\ldots +(2m-1)\alpha_{m})=$$
$$=\dim (Z(W))= s(\tau_{1}^{2}+\ldots+\tau_{r}^{2}).$$
Notice that when $\deg (p)=s=1$ the result  matches the Frobenius formula of the dimension of the centralizer of a Jordan and Weyr form (\cite{Omeara}).

\section*{Acknowledgements}
The second author is partially supported by grant MTM2015-65361-P
MINECO/FEDER, UE. The third author is partially supported by grant MTM2017-83624-P MINECO.

\medskip

\bibliographystyle{model1a-num-names}

\end{document}